\documentclass[a4paper,12pt]{article}
\usepackage{ucs}
\usepackage{amsfonts, amssymb, amsmath, amsthm, amscd}
\usepackage{graphicx}
\usepackage{cite}
\ifx\undefined \pdfgentounicode \else
\input{glyphtounicode} \pdfgentounicode=1
\fi

\author{A.A. Vasil'eva\footnote{Faculty of Mechanics and Mathematics, Lomonosov Moscow State University; Moscow Center for Fundamental and Applied Mathematics.}}
\title{Kolmogorov widths of an intersection of anisotropic finite-dimensional balls in $l_q^k$ for $1\le q\le 2$}
\date{}
\begin{document}

\maketitle

\newenvironment{Biblio}{%
                  \renewcommand{\refname}{\footnotesize REFERENCES}%
                  \begin{thebibliography}{99}%
                  \itemsep -0.2em}%
                  {\end{thebibliography}}

\def\inff{\mathop{\smash\inf\vphantom\sup}}
\renewcommand{\le}{\leqslant}
\renewcommand{\ge}{\geqslant}
\newcommand{\sgn}{\mathrm {sgn}\,}
\newcommand{\inter}{\mathrm {int}\,}
\newcommand{\dist}{\mathrm {dist}}
\newcommand{\supp}{\mathrm {supp}\,}
\newcommand{\R}{\mathbb{R}}
\newcommand{\Z}{\mathbb{Z}}
\newcommand{\N}{\mathbb{N}}
\newcommand{\Q}{\mathbb{Q}}
\theoremstyle{plain}
\newtheorem{Trm}{Theorem}
\newtheorem{trma}{Theorem}
\newtheorem{Def}{Definition}
\newtheorem{Cor}{Corollary}
\newtheorem{Lem}{Lemma}
\newtheorem{Rem}{Remark}
\newtheorem{Sta}{Proposition}
\newtheorem{Exa}{Example}
\renewcommand{\proofname}{\bf Proof}
\renewcommand{\thetrma}{\Alph{trma}}

\begin{abstract}
In this paper, order estimates for the Kolmogorov $n$-widths of an intersection of an arbitrary family of balls $\nu_\alpha B^{\overline{k}}_{\overline{p}_\alpha}$ in $l_q^k$ are obtained for $1\le q\le 2$, $n\le \frac k2$. Here $\overline{p}_\alpha = (p_{\alpha,1}, \, \dots, \, p_{\alpha,d})$, $\overline{k}=(k_1, \, \dots, \, k_d)$, $k=k_1\dots k_d$, $B^{\overline{k}}_{\overline{p}_\alpha}$ is the unit ball with respect to the anisotropic norm given by the vector $\overline{p}_\alpha$.
\end{abstract}

\section{Introduction}

In this paper, we study the problem of order estimates for the Kolmogorov widths of an intersection of a family of balls of different radii with respect to anisotropic norms in the space $l_q^k$ with $1\le q\le 2$. This question appears naturally in the context of the studies begun in \cite{galeev1, mal_rjut, vas_ball_inters, vas_mix2,  vas_mix_sev, vas_anisotr, mal_rjut1}. The result can be applied, for example, in estimating the widths of an intersection of anisotropic periodic Sobolev classes.

Let us give necessary notation.

Let $N\in \N$, $1\le s\le \infty$, $(x_i)_{i=1}^N\in \mathbb{R}^N$. We set $\|(x_i)_{i=1}^N\|_{l_s^N} = \left(\sum \limits _{i=1}^N |x_i|^s\right)^{1/s}$ for $s<\infty$, $\|(x_i)_{i=1}^N\|_{l_s^N} = \max _{1\le i\le N}|x_i|$ for $s=\infty$. The space $\mathbb{R}^N$ with this norm is denoted by $l_s^N$. By $B_s^N$ we denote the unit ball in $l_s^N$.

Let now $k_1, \, \dots, \, k_d\in \N$, $1\le p_1, \, \dots, \, p_d\le \infty$, $\overline{k}=(k_1, \, \dots, \, k_d)$, $\overline{p}=(p_1, \, \dots, \, p_d)$. By $l_{\overline{p}}^{\overline{k}}$ we denote the space $\mathbb{R}^{k_1\dots k_d}= \{(x_{j_1,\dots,j_d})_{1\le j_s\le k_s, \, 1\le s\le d}:\; x_{j_1,\dots,j_d}\in \mathbb{R}\}$ with norm defined by induction on $d$: for $d=1$ it is $\|\cdot\|_{l_{p_1}^{k_1}}$, and for $d\ge 2$,
$$
\|(x_{j_1,\dots,j_d})_{1\le j_s\le k_s, \, 1\le s\le d}\|_{l_{\overline{p}}^{\overline{k}}} = \left\|\bigl(\|(x_{j_1,\dots, \, j_{d-1}, \, j_d})_{1\le j_s\le k_s, \, 1\le s\le d-1}\|_{l_{(p_1,\dots, \, p_{d-1})}^{(k_1,\dots, \, k_{d-1})}}\bigr)_{j_d=1}^{k_d}\right\|_{l_{p_d}^{k_d}}.
$$
Given $\overline{k}=(k_1, \, \dots, \, k_d)$, in what follows we write $k = k_1\dots k_d$.

\begin{Def}
Let $X$ be a normed space, and let $M\subset X$, $n\in
\Z_+$. The Kolmogorov $n$-width of the set $M$ in $X$ is defined by the formula
$$
d_n(M, \, X) = \inf _{L\in {\cal L}_n(X)} \sup _{x\in M} \inf
_{y\in L} \|x-y\|;
$$
here ${\cal L}_n(X)$ is the family of all subsets in $X$
of dimension at most $n$.
\end{Def}

Let $A$ be a nonempty set. Suppose that, for each $\alpha\in A$, a number $\nu_\alpha>0$ and a vector $\overline{p}_\alpha=(p_{\alpha,1}, \, \dots, \, p_{\alpha,d})$ are given; here $1\le p_{\alpha,j}\le \infty$, $1\le j\le d$. Let $\overline{k}\in \N^d$. We set
\begin{align}
\label{m_def} M = \cap _{\alpha\in A} \nu_\alpha B^{\overline{k}}_{\overline{p}_\alpha}.
\end{align}
In this paper we obtain order estimates for the Kolmogorov $n$-widths of the set
$M$ in the space $l_q^k$ for $1\le q\le 2$, $n\le \frac k2$.

The estimates for the Kolmogorov widths of the ball $B_p^N$ in $l_q^N$ were obtained by Pietsch, Stesin, Ismagilov, Kashin, Gluskin, Garnaev \cite{pietsch1, stesin, kashin_oct, bib_kashin, gluskin1, bib_gluskin, garn_glus, bib_ismag}. For details, see \cite{itogi_nt, kniga_pinkusa, teml_book, alimov_tsarkov}. Estimates for the widths of an intersection of balls $\cap _{\alpha \in A} \nu_\alpha B^N_{p_\alpha}$ were obtained \cite{galeev1, vas_ball_inters}.

In \cite{galeev2, galeev5, izaak1, izaak2, mal_rjut, vas_besov, dir_ull} the problem of estimates for the widths of the ball $B^{k_1,k_2}_{p_1,p_2}$ in $l^{k_1,k_2}_{q_1,q_2}$ was studied (i.e., the case of anisotropic norms with $d=2$ was considered); for details, see \cite{vas_mix2}. In addition, recently Malykhin and Rjutin \cite{mal_rjut1} obtained estimates for the widths of the ball $B^{k_1,k_2}_{p_1,p_2}$ in $l^{k_1,k_2}_{q,q}$ for $1\le q\le 2$ and arbitrary $p_1$, $p_2$. In \cite{vas_mix_sev} order estimates for the widths of the set $\cap _{\alpha\in A}\nu_\alpha B^{k_1,k_2}_{p_{\alpha,1},p_{\alpha,2}}$ in $l^{k_1,k_2}_{q_1,q_2}$ for $2\le q_1, \, q_2<\infty$ and arbitrary $p_{\alpha,1}$, $p_{\alpha,2}$ are obtained. Estimates of widths of these sets have applications in problems of approximation of Besov classes and their intersections.

In \cite{vas_anisotr} estimates for the widths of $B^{\overline{k}}_{\overline{p}}$ in $l^{\overline{k}}_{\overline{q}}$ for $2\le q_j<\infty$, $j=1, \, \dots, \, d$, and arbitrary $\overline{p}$ are obtained; the result was applied in solving the problem of widths of anisotropic periodic Sobolev classes in an anisotropic Lebesgue space.

Now we formulate the main result of the present paper. First we need some notation.

Let $X$, $Y$ be sets, $f_1$, $f_2:\ X\times Y\rightarrow \mathbb{R}_+$.
We write $f_1(x, \, y)\underset{y}{\lesssim} f_2(x, \, y)$ (or
$f_2(x, \, y)\underset{y}{\gtrsim} f_1(x, \, y)$) if for each
$y\in Y$ there is $c(y)>0$ such that $f_1(x, \, y)\le
c(y)f_2(x, \, y)$ for all $x\in X$; $f_1(x, \,
y)\underset{y}{\asymp} f_2(x, \, y)$ if $f_1(x, \, y)
\underset{y}{\lesssim} f_2(x, \, y)$ and $f_2(x, \,
y)\underset{y}{\lesssim} f_1(x, \, y)$.

Given $a\in \R$, we write $a_+ = \max\{a, \, 0\}$. For $\overline{p}=(p_1, \, \dots, \, p_d)$, $1\le q\le 2$, $\overline{k}=(k_1, \, \dots, \, k_d)$ we set
\begin{align}
\label{phi_def} \Phi(\overline{p}, \, \overline{k}, \, q) = \prod _{j=1}^d k_j^{(1/q-1/p_j)_+}.
\end{align}
From the results of \cite{mal_rjut1} it follows that
$$
d_n(B^{\overline{k}}_{\overline{p}}, \, l^k_q) \underset{q}{\asymp} \Phi(\overline{p}, \, \overline{k}, \, q), \quad 1\le q\le 2, \, n\le k/2;
$$
for details, see \S 2.

For $\overline{p} = (p_1, \, \dots, \, p_d)$, $\lambda\in \R$, we denote by $\frac{\lambda}{\overline{p}}$ the vector with coordinates $\left(\frac{\lambda}{p_1}, \, \dots,\, \frac{\lambda}{p_d}\right)$.

Let $I=\{i_1, \, \dots, \, i_l\}\subset \{1, \, \dots, \, d\}$ be a nonempty set, $i_1<\dots<i_l$. Given $x=(x_1, \, \dots, \, x_d)\in \R^d$, we set $x_I=(x_{i_1}, \, \dots, \, x_{i_l})\in \R^l$. 

\begin{Def}
\label{nm_def}
Let $1\le m\le d+1$, $\overline{\alpha} = (\alpha_1, \, \dots, \, \alpha_m) \in A^m$. We say that $\overline{\alpha}\in {\cal N}_m$ if there are a set $I\subset \{1, \, \dots, \, d\}$ and numbers $\lambda_j=\lambda_j(\overline{\alpha}, \, I)>0$, $j=1, \, \dots, \, m$, such that $\sum \limits _{j=1}^m \lambda_j=1$, $\# I=m-1$, 
\begin{align}
\label{sum_paj_i_q}
\sum \limits _{j=1}^m \frac{\lambda_j}{p_{\alpha_j,i}} = \frac 1q, \quad i\in I,
\end{align}
and the points $(1/\overline{p}_{\alpha_j})_{I}$, $j=1, \, \dots, \, m$, are affinely independent. In this case we define the vector $\overline{\theta}(\overline{\alpha}, \, I) = (\theta_1(\overline{\alpha}, \, I), \, \dots, \, \theta_d(\overline{\alpha}, \, I))$ by the equation
\begin{align}
\label{theta_a_def}
\frac{1}{\overline{\theta}(\overline{\alpha}, \, I)} = \sum \limits _{j=1}^m \frac{\lambda_j}{\overline{p}_{\alpha_j}}.
\end{align}
\end{Def}

\begin{Rem}
\label{rem1}
By \eqref{sum_paj_i_q}, for $i\in I$ we have $\theta_i(\overline{\alpha}, \, I) = q$.
\end{Rem}

\begin{Trm}
\label{main}
Let $d\in \N$, $\overline{k}=(k_1, \, \dots, \, k_d)\in \N^d$, $k=k_1\dots k_d$, $1\le q\le 2$, $n\le \frac k2$. Let $A$ be a nonempty set, $\nu_\alpha>0$, $\overline{p}_\alpha = (p_{\alpha,1}, \, \dots, \, p_{\alpha,d}) \in [1, \, \infty]^d$, $\alpha\in A$. The set $M$ is defined by formula \eqref{m_def}, and the function $\Phi$, by formula \eqref{phi_def}; the sets ${\cal N}_m$ $(1\le m \le d+1)$, the subsets $I$, the numbers $\lambda_j(\overline{\alpha}, \, I)$ and the vector $\overline{\theta}(\overline{\alpha}, \, I)$ are such as in Definition {\rm \ref{nm_def}}. Then
$$
d_n(M, \, l_q^k) \underset{q,d}{\asymp} \min _{1\le m\le d+1} \inf _{\overline{\alpha} \in {\cal N}_m, \, I} \nu_{\alpha_1}^{\lambda_1(\overline{\alpha}, \, I)} \dots \nu_{\alpha_m}^{\lambda_m(\overline{\alpha}, \, I)} \Phi(\overline{\theta}(\overline{\alpha}, \, I), \, \overline{k}, \, q).
$$
\end{Trm}

\section{Auxilliary assertions and the upper estimate}

The following result is obtained in \cite{mal_rjut1}.

\begin{trma}
\label{mal_rjut_trm} {\rm (see \cite{mal_rjut1}, Theorem 1).} Let $S_k$ be the group of permutations of $k$ elements, and let $H$ be a subgroup of $S_k$ that transitively acts on $\{1, \, \dots, \, k\}$ $($i.e., for all $i$, $j\in \{1, \, \dots, \, k\}$ there is a permutation $h\in H$ such that $h(i)=j)$. Let $K\subset \R^k$ be an unconditional set invariant under the action of $H$ on the set of coordinates; i.e., $(\varepsilon_1x_{h(1)}, \, \dots, \, \varepsilon_k x_{h(k)})\in K$ for each $(x_1, \, \dots, \, x_k)\in K$, $(\varepsilon_1, \, \dots, \, \varepsilon_k)\in \{\pm 1\}^k$, $h\in H$. Then, for $1\le q\le 2$, $\varepsilon\in (0, \, 1)$, $n\le (1-\varepsilon)k$, the following estimate holds:
$$
d_n(K, \, l_q^k) \underset{q, \, \varepsilon}{\gtrsim} \sup _{x\in K} \|x\|_{l_q^k}.
$$
\end{trma}

For $1<q\le 2$, the proof of Theorem \ref{mal_rjut_trm}
exploits some ideas of \cite{bib_glus_3}. For $q=1$, the estimate was obtained in \cite{mal}.

The set $M$ from formula \eqref{m_def} satisfies the conditions of Theorem \ref{mal_rjut_trm}. As $H$, we take the subgroup of elements of form $\sigma = (\sigma_1, \, \dots, \, \sigma_d)\in S_{k_1} \times \dots \times S_{k_d}$; it acts according to the equation $\sigma(i_1, \, \dots, \, i_d) = (\sigma_1(i_1), \, \dots, \, \sigma_d(i_d))$, $i_j\in \{1, \, \dots, \, k_j\}$, $1\le j\le d$.

\begin{Cor}
\label{1_ball}
Let $1\le q\le 2$, $n\le \frac k2$, anf let the set $M$ be defined by formula \eqref{m_def}. Then 
\begin{align}
\label{m_rigid}
d_n(M, \, l_q^k) \underset{q}{\gtrsim} \sup _{x\in M} \|x\|_{l_q^k}.
\end{align}
In particular, if $M=B_{\overline{p}}^{\overline{k}}$, we get
\begin{align}
\label{fin_dim_ball_phi}
\Phi(\overline{p}, \, \overline{k}, \, q)\underset{q}{\lesssim}d_n(B_{\overline{p}}^{\overline{k}}, \, l_q^k) \le \Phi(\overline{p}, \, \overline{k}, \, q),
\end{align}
where the function $\Phi$ is defined by formula \eqref{phi_def}.
\end{Cor}

The upper estimate in \eqref{fin_dim_ball_phi} is trivial; the lower estimate follows from \eqref{m_rigid}. For $d=2$, formula \eqref{fin_dim_ball_phi} was written in \cite{mal_rjut1}. 

\begin{Lem}
\label{inter_incl} Let $\overline{k}=(k_1, \, \dots, \, k_d)\in \N^d$, $m\in \N$, $\nu_j>0$, $\overline{p}^j\in [1, \, \infty]^d$, $\lambda_j\ge 0$, $1\le j\le m$, $\sum \limits _{j=1}^m \lambda_j=1$, and let the vector $\overline{p}$ be given by the equation
$\frac{1}{\overline{p}} = \sum \limits _{j=1}^m \frac{\lambda_j}{\overline{p}^j}$.
Then
$$
\cap _{j=1}^m \nu_j B^{\overline{k}}_{\overline{p}^j} \subset \nu_1^{\lambda_1} \dots \nu_m^{\lambda_m}B^{\overline{k}}_{\overline{p}}.
$$
\end{Lem}
\begin{proof}
The case $m=1$ is trivial. The case $m=2$ follows from H\"{o}lder's inequality; for $d=2$ it is proved in \cite[Lemma 1]{vas_mix2}, for arbitrary $d\in \N$ the proof is similar. For arbitrary $m\ge 3$ the inclusion can be obtained by induction in $m$.
\end{proof}

From the upper estimate in \eqref{fin_dim_ball_phi} and Lemma \ref{inter_incl} we get
\begin{Cor}
Let the conditions of Theorem {\rm \ref{main}} hold. Then $$d_n(M, \, l_q^k)\le \min _{1\le m\le d+1} \inf _{\overline{\alpha} \in {\cal N}_m, \, I} \nu_{\alpha_1}^{\lambda_1(\overline{\alpha}, \, I)} \dots \nu_{\alpha_m}^{\lambda_m(\overline{\alpha}, \, I)} \Phi(\overline{\theta}(\overline{\alpha}, \, I), \, \overline{k}, \, q).$$
\end{Cor}

\section{Proof of the lower estimate: the case of finite $A$ and general position}

In this section, we consider the case $\# A<\infty$; in addition, we suppose that the points $(\overline{p}_{\alpha})_{\alpha\in A}$ are in general position (see the definition below).

\begin{Def}
\label{gen_pos_def}
We say that the points $(\overline{p}_{\alpha})_{\alpha\in A}$ are in general position if
\begin{enumerate}
\item for each number $m\in \{2, \, \dots, \, d+1\}$, for each subset $I\subset \{1, \, \dots, \, d\}$ such that $\# I = m-1$, and for all different $\alpha_1, \, \dots, \, \alpha_m\in A$ the points $\{(1/\overline{p}_{\alpha_j})_I\}_{j=1}^m$ are affinely independent;

\item if $1\le m\le d+1$, $\overline{\alpha}=(\alpha_1, \, \dots, \, \alpha_m) \in {\cal N}_m$, the subset $I$ is such as in Definition {\rm \ref{nm_def}}, $i\notin I$, then $\theta_i(\overline{\alpha})\ne q$; in particular, for $m=1$ we get $p_{\alpha,i}\ne q$ for each $\alpha\in A$, $i=1, \, \dots, \, d$.
\end{enumerate}
\end{Def}

We set
\begin{align}
\label{psi_def}
\Psi(\{\overline{p}_{\alpha}\}_{\alpha\in A}, \, \overline{k}, \, q) = \min _{1\le m\le d+1} \inf _{\overline{\alpha} \in {\cal N}_m, \, I} \nu_{\alpha_1}^{\lambda_1(\overline{\alpha}, \, I)} \dots \nu_{\alpha_m}^{\lambda_m(\overline{\alpha}, \, I)} \Phi(\overline{\theta}(\overline{\alpha}, \, I), \, \overline{k}, \, q),
\end{align}
where $I$, $\overline{\theta}(\overline{\alpha}, \, I)$ are from Definition \ref{nm_def}. Since $A$ in finite, the minimum in \eqref{psi_def} is attained.

Let $\Psi(\{\overline{p}_{\alpha}\}_{\alpha\in A}, \, \overline{k}, \, q) = \nu_{\alpha_1}^{\lambda_1(\overline{\alpha}, \, I)} \dots \nu_{\alpha_m}^{\lambda_m(\overline{\alpha}, \, I)} \Phi(\overline{\theta}(\overline{\alpha}, \, I), \, \overline{k}, \, q)$. We show that $$d_n(M, \, l_q^k) \underset{q,d}{\gtrsim} \nu_{\alpha_1}^{\lambda_1(\overline{\alpha}, \, I)} \dots \nu_{\alpha_m}^{\lambda_m(\overline{\alpha}, \, I)} \Phi(\overline{\theta}(\overline{\alpha}, \, I), \, \overline{k}, \, q).$$ By Corollary \ref{1_ball}, it suffices to check that
\begin{align}
\label{sup_x_m}
\sup _{x\in M} \|x\|_{l_q^k} \underset{q}{\gtrsim} \nu_{\alpha_1}^{\lambda_1(\overline{\alpha}, \, I)} \dots \nu_{\alpha_m}^{\lambda_m(\overline{\alpha}, \, I)} \Phi(\overline{\theta}(\overline{\alpha}, \, I), \, \overline{k}, \, q).
\end{align}

Given $\overline{s}=(s_1, \, \dots, \, s_d)$ with $s_i\in \{1, \, \dots, \, k_i\}$, $1\le i\le d$, we set
\begin{align}
\label{xs_def}
\hat x(\overline{s}) =(\hat x_{i_1, \, \dots, \, i_d}(\overline{s}))_{1\le i_j\le k_j, \, 1\le j\le d}, \text{ where }\hat x_{i_1, \, \dots, \, i_d}(\overline{s})= \begin{cases} 1, & 1\le i_j \le s_j, \; j=1, \, \dots, \, d, \\ 0, & \text{otherwise}.\end{cases}
\end{align}

Let us prove \eqref{sup_x_m}.

{\bf The case $m=1$.} We have
\begin{align}
\label{m1min} \Psi(\{\overline{p}_{\beta}\}_{\beta\in A}, \, \overline{k}, \, q)=\nu_\alpha \Phi(\overline{p}_\alpha, \, \overline{k}, \, q) \stackrel{\eqref{phi_def}}{=} \nu_\alpha \prod_{i=1}^d k_i^{(1/q-1/p_{\alpha,i})_+}
\end{align}
for some $\alpha \in A$.

Let $J=\{i\in \overline{1, \, d}:\; p_{\alpha,i}>q\}$,
\begin{align}
\label{s_i} s_i = \begin{cases} k_i, & i\in J, \\ 1, & i\notin J.\end{cases}
\end{align}
We show that
\begin{align}
\label{incl}
\nu_\alpha s_1^{-1/p_{\alpha,1}} \dots s_d^{-1/p_{\alpha,d}} \hat x(\overline{s}) \in M;
\end{align}
then, by \eqref{xs_def}, \eqref{m1min} and \eqref{s_i}, we get \eqref{sup_x_m}.

In order to prove \eqref{incl} it suffices to check that, for each $\beta \in A$,
\begin{align}
\label{incl_ineq} \nu_\alpha \prod _{i=1}^d s_i^{1/p_{\beta,i}-1/p_{\alpha,i}} \le \nu_\beta.
\end{align}

We set 
\begin{align}
\label{d_def}
D=\{(1/p_1, \, \dots , \, 1/p_d)\in [0, \, 1]^d:\; p_j>q, \, j\in J, \, p_j<q, \, j\notin J\}.
\end{align}
From the definition of $J$ and generality of position it follows that $(1/p_{\alpha,1}, \, \dots , \, 1/p_{\alpha,d})\in D$.

We set
\begin{align}
\label{lam_def}
\lambda =\sup \Bigl\{\mu\in [0, \, 1]:\; \frac{1-\mu}{\overline{p}_\alpha}+ \frac{\mu}{\overline{p}_\beta} \in D\Bigr\}.
\end{align}
Notice that $\lambda>0$. We define $\overline{\theta}\in [1, \, \infty]^d$ by the equation $\frac{1}{\overline{\theta}} = \frac{1-\lambda}{\overline{p}_{\alpha}}+ \frac{\lambda}{\overline{p}_{\beta}}$. If $\lambda=1$, then $\frac{1}{\overline{p}_\beta}\in D$ (it follows from generality of position), and, by \eqref{phi_def}, \eqref{psi_def}, \eqref{m1min}, we get
$$
\nu_\alpha \prod _{i\in J} k_i^{1/q-1/p_{\alpha,i}} \le \nu_\beta \prod _{i\in J} k_i^{1/q-1/p_{\beta,i}}, 
$$
which implies \eqref{incl_ineq}. If $\lambda<1$, then $(\alpha, \, \beta) \in {\cal N}_2$ by \eqref{d_def}, \eqref{lam_def} and Definitions \ref{nm_def}, \ref{gen_pos_def}; from \eqref{lam_def} it follows that $\frac{1}{\overline{\theta}}\in \overline{D}$. This together with \eqref{m1min} implies that $\nu_\alpha \Phi(\overline{p}_\alpha, \overline{k}, \, q) \le \nu_\alpha^{1-\lambda}\nu_\beta^\lambda \Phi(\overline{\theta}, \overline{k}, \, q)$. From \eqref{phi_def} and \eqref{d_def} we get
$$
\nu_\alpha \prod _{i\in J} k_i^{1/q-1/p_{\alpha,i}} \le \nu_\alpha^{1-\lambda}\nu_\beta^\lambda \prod _{i\in J} k_i^{1/q-(1-\lambda)/p_{\alpha,i}-\lambda/p_{\beta,i}},
$$
which yields \eqref{incl_ineq}.

{\bf The case $m>1$.} In what follows, for brevity, we write $\lambda_j(\overline{\alpha}):=\lambda_j(\overline{\alpha}, \, I)$, $\overline{\theta}(\overline{\alpha}):= \overline{\theta}(\overline{\alpha}, \, I)$, where $I$ is the set from Definition \ref{nm_def}. We have (see \eqref{phi_def}, \eqref{psi_def})
\begin{align}
\label{min_mg1} \Psi(\{\overline{p}_{\alpha}\}_{\alpha\in A}, \, \overline{k}, \, q) = \nu_{\alpha_1}^{\lambda_1(\overline{\alpha})} \dots \nu_{\alpha_m}^{\lambda_m(\overline{\alpha})} \prod_{i=1}^d k_i^{(1/q-1/\theta_i(\overline{\alpha}))_+}.
\end{align}

Let $T_+=\{i\notin I:\; \theta_i( \overline{\alpha})>q\}$, $T_-=\{i\notin I:\; \theta_i( \overline{\alpha})<q\}$. By generality of position, $\{1, \, \dots, \, d\}\backslash I = T_+ \sqcup T_-$.

We define the vector $\overline{s}=(s_1, \, \dots, \, s_d)$ as follows. For $i \notin I$ we set
\begin{align}
\label{s_inii} s_i = \begin{cases} k_i, & i\in T_+, \\ 1, & i\in T_-.\end{cases}
\end{align}
For $i\in I$ the numbers $s_i$ are defined by the system of equations
\begin{align}
\label{s_i_syst} \left\{\begin{array}{l} \frac{\nu_{\alpha_2}}{\nu_{\alpha_1}} = \prod _{i\in T_+} k_i^{1/p_{\alpha_2,i}-1/p_{\alpha_1,i}} \prod _{i\in I} s_i^{1/p_{\alpha_2,i}-1/p_{\alpha_1,i}}, \\ \dots \\ \frac{\nu_{\alpha_m}}{\nu_{\alpha_1}} = \prod _{i\in T_+} k_i^{1/p_{\alpha_m,i}-1/p_{\alpha_1,i}} \prod _{i\in I} s_i^{1/p_{\alpha_m,i}-1/p_{\alpha_1,i}}. \end{array} \right.
\end{align}
By generality of position, the points $(1/\overline{p}_{\alpha_j})_I$ $(1\le j\le m)$ are affinely independent; hence the vectors $(1/p_{\alpha_j}-1/p_{\alpha_1})_I$ $(2\le j\le m)$ are linearly independent. Taking the logarithm in the equations \eqref{s_i_syst}, we see that the numbers $s_i$ ($i\in I$) are well-defined.

Let $\mu_{j,k}\ge 0$, $1\le j\le m$, $\sum \limits _{j=1}^m \mu_{j,k}=1$, $k=1, \, 2$. We define the vectors $\overline{\theta}_k$ and the numbers $\nu_{(k)}$ by the equations
\begin{align}
\label{1theta_k}
\frac{1}{\overline{\theta}_k} =\sum \limits _{j=1}^m \frac{\mu_{j,k}}{\overline{p}_{\alpha_j}}, \quad \nu_{(k)} = \prod _{j=1}^m \nu_{\alpha_j}^{\mu_{j,k}}, \quad k=1, \, 2.
\end{align}
From \eqref{s_i_syst} it follows that
\begin{align}
\label{nu1_nu2} \frac{\nu_{(1)}}{\nu_{(2)}} = \frac{\prod _{l\in T_+} k_l^{1/\theta_{1,l}} \prod _{l\in I} s_l^{1/\theta_{1,l}}}{\prod _{l\in T_+} k_l^{1/\theta_{2,l}} \prod _{l\in I} s_l^{1/\theta_{2,l}}} \stackrel{\eqref{s_inii}}{=}  \frac{\prod _{l=1}^d s_l^{1/\theta_{1,l}}}{\prod _{l=1}^d s_l^{1/\theta_{2,l}}}.
\end{align}

We show that $1\le s_i\le k_i$, $i\in \{1, \, \dots, \, d\}$. By \eqref{s_inii}, it suffices to consider the case $i\in I$.

We denote
\begin{align}
\label{sigma_def}
\Sigma = {\rm conv}\, (1/\overline{p}_{\alpha_j})_{1\le j\le m}.
\end{align}

Let $i\in I$. Consider the segment 
\begin{align}
\label{delta_otr}
\Delta=\{(\xi_l)_{1\le l\le m}\subset \Sigma:\; \xi_l=1/q, \, l\in I\backslash\{ i\}\}.
\end{align}
Its endpoints are $1/\overline{\theta}_k\in \Sigma$, $k=1, \, 2$. Since $\{(1/\overline{p}_{\alpha_j})_I\}_{j=1}^m$ are affinely independent and $\lambda_j(\overline{\alpha})>0$, $1\le j\le m$ (see Definition \ref{nm_def}), we have $\theta_{k,i}\ne q$, $k=1, \, 2$; i.e., $1/\overline{\theta}(\overline{\alpha})$ is the interior point of $\Delta$. Without loss of generality, 
\begin{align}
\label{theta12i}
1/\theta_{1,i}<1/q<1/\theta_{2,i}.
\end{align}

By \eqref{sigma_def}, there are $\mu_{j,k}\ge 0$, $1\le j\le m$, $k=1, \, 2$, such that $\sum \limits_{j=1}^m \mu_{j,k}=1$ and the first equality in \eqref{1theta_k} holds; the numbers $\nu_{(k)}$ are defined by the second equation of \eqref{1theta_k}. Further, there is $\lambda\in (0, \, 1)$ such that
\begin{align}
\label{theta_a_12}
\frac{1}{\overline{\theta}(\overline{\alpha})} = \frac{1-\lambda}{\overline{\theta}_1} + \frac{\lambda}{\overline{\theta}_2}.
\end{align}
By affine independence of $\{(1/\overline{p}_{\alpha_j})_I\}_{j=1}^m$, this yields that $\lambda_j(\overline{\alpha}) = (1-\lambda)\mu_{j,1}+\lambda \mu_{j,2}$; hence $\nu_{\alpha_1}^{\lambda_1(\overline{\alpha})} \dots \nu_{\alpha_m}^{\lambda_m(\overline{\alpha})} = \nu_{(1)}^{1-\lambda} \nu_{(2)}^\lambda$.

We set
$$
\mu_1= \inf \Bigr\{\mu\in [0, \, \lambda]:\; \frac{1-\mu}{\theta_{1,l}} + \frac{\mu}{\theta_{2,l}}< \frac 1q, \; l\in T_+, \; \frac{1-\mu}{\theta_{1,l}} + \frac{\mu}{\theta_{2,l}}> \frac 1q, \; l\in T_-\Bigl\},
$$
$$
\mu_2= \sup \Bigr\{\mu\in [\lambda, \, 1]:\; \frac{1-\mu}{\theta_{1,l}} + \frac{\mu}{\theta_{2,l}}< \frac 1q, \; l\in T_+, \; \frac{1-\mu}{\theta_{1,l}} + \frac{\mu}{\theta_{2,l}}> \frac 1q, \; l\in T_-\Bigl\}.
$$
Then $0\le \mu_1<\lambda<\mu_2\le 1$. We define the vectors $\overline{\sigma}_1$, $\overline{\sigma}_2$ by the equations
\begin{align}
\label{sigma_k_def}
\frac{1}{\overline{\sigma}_k} = \frac{1-\mu_k}{\overline{\theta}_1} + \frac{\mu_k}{\overline{\theta}_2}, \quad k = 1, \, 2.
\end{align}

From \eqref{phi_def}, \eqref{psi_def}, \eqref{min_mg1} and Definitions \ref{nm_def}, \ref{gen_pos_def} it follows that 
\begin{align}
\label{nu11lnu2lmut}
\nu_{(1)}^{1-\lambda} \nu_{(2)}^\lambda \prod _{l\in T_+} k_l^{1/q-1/\theta_l(\overline{\alpha})} \le \nu_{(1)} ^{1-\mu_t} \nu_{(2)}^{\mu_t} \Phi(\overline{\sigma}_t, \, \overline{k}, \, q), \quad t=1, \, 2.
\end{align}
From the definition of $\mu_t$ and \eqref{theta12i}, \eqref{sigma_k_def} it follows that $$\Phi(\overline{\sigma}_1, \, \overline{k}, \, q) = \prod _{l\in T_+} k_l^{1/q-1/\sigma_{1,l}} k_i^{1/q-1/\sigma_{1,i}}, \quad \Phi(\overline{\sigma}_2, \, \overline{k}, \, q) = \prod _{l\in T_+} k_l^{1/q-1/\sigma_{2,l}}.$$ This together with \eqref{theta_a_12}, \eqref{sigma_k_def}, \eqref{nu11lnu2lmut} and the equality $q=\theta_i(\overline{\alpha})$ (see Remark \ref{rem1}) implies that
$$
\frac{\nu_{(1)}^{\lambda-\mu_1}}{\nu_{(2)}^{\lambda-\mu_1}} \ge \prod_{l\in T_+} k_l^{(\lambda-\mu_1)(1/\theta_{1,l}-1/\theta_{2,l})} k_i^{(\lambda-\mu_1)(1/\theta_{1,i}-1/\theta_{2,i})},
$$
$$
\frac{\nu_{(1)}^{\mu_2-\lambda}}{\nu_{(2)}^{\mu_2-\lambda}} \le \prod_{l\in T_+} k_l^{(\mu_2-\lambda)(1/\theta_{1,l}-1/\theta_{2,l})}.
$$
Taking into account \eqref{nu1_nu2} and the equalities $\theta_{1,l}=\theta_{2,l} = q$ for $l\in I\backslash\{i\}$ (see \eqref{delta_otr}), we get
$$
\prod _{l\in T_+} k_l^{1/\theta_{1,l}-1/\theta_{2,l}} s_i ^
{1/\theta_{1,i}-1/\theta_{2,i}} \ge \prod_{l\in T_+} k_l^{1/\theta_{1,l}-1/\theta_{2,l}} k_i^{1/\theta_{1,i}-1/\theta_{2,i}},
$$
$$
\prod _{l\in T_+} k_l^{1/\theta_{1,l}-1/\theta_{2,l}} s_i ^
{1/\theta_{1,i}-1/\theta_{2,i}} \le \prod_{l\in T_+} k_l^{1/\theta_{1,l}-1/\theta_{2,l}};
$$
this together with \eqref{theta12i} yields the desired inequality $1\le s_i\le k_i$.

Now we prove that
$$ \nu_{\alpha_1}^{\lambda_1(\overline{\alpha})} \dots \nu_{\alpha_m}^{\lambda_m(\overline{\alpha})} s_1^{-1/\theta_1(\overline{\alpha})} \dots s_d^{-1/\theta_d(\overline{\alpha})}\hat x(\overline{s}) \in M;
$$
i.e. (see \eqref{xs_def}),
\begin{align}
\label{incl_mg1_ineq} \nu_{\alpha_1}^{\lambda_1 (\overline{\alpha})} \dots \nu_{\alpha_m}^{\lambda_m (\overline{\alpha})} s_1^{1/p_{\beta,1}-1/\theta_1 (\overline{\alpha})} \dots s_d^{1/p_{\beta,d}-1/\theta_d(\overline{\alpha})} \le \nu_\beta, \quad \beta \in A.
\end{align}
Then by \eqref{xs_def}, \eqref{min_mg1}, \eqref{s_inii} and the equality $\theta_i(\overline{\alpha})=q$ for $i\in I$ we get \eqref{sup_x_m}.

If $\beta \in \{\alpha_1, \, \dots, \, \alpha_m\}$, then \eqref{incl_mg1_ineq} follows from \eqref{nu1_nu2}.

Consider the case $\beta \notin \{\alpha_1, \, \dots, \, \alpha_m\}$.

\begin{Lem}
\label{conv_hull}
Let $\xi_1, \, \dots, \, \xi_m \in \R^{m-1}$ be affinely independent points, $\eta \in \R^{m-1}$. Suppose that for each $j\in \{1, \, \dots, \, m\}$ the points $\{\xi_1, \, \dots, \, \xi_{j-1}, \, \xi_{j+1}, \, \dots, \, \xi_m, \, \eta\}$ are affinely independent. We set $\Delta = {\rm conv}\, \{\xi_1, \, \dots, \, \xi_m\}$. Let $a\in {\rm int}\, \Delta$. Then there is $i\in \{1, \, \dots, \, m\}$ such that $a\in {\rm conv}\, \{\xi_1, \, \dots, \, \xi_{i-1}, \, \xi_{i+1}, \, \dots, \, \xi_m, \, \eta\}$.
\end{Lem}
\begin{proof}
If $\eta \in {\rm int}\, \Delta$, then $\Delta$ is divided into nonoverlapping simplices $${\rm conv}\, \{\xi_1, \, \dots, \, \xi_{i-1}, \, \xi_{i+1}, \, \dots, \, \xi_m, \, \eta\}, \quad i=1, \, \dots, \, m.$$

Let $\eta \notin {\rm int}\, \Delta$. Then $\eta \notin \Delta$. We prove the Lemma by induction on $m$. For $m=2$ it can be easily checked. Now we make the induction step from $m-1$ to $m$. Denote by $L_j$ ($j=1, \, \dots, \, m$) the affine span of $\xi_1, \, \dots, \, \xi_{j-1}, \, \xi_{j+1}, \, \dots, \, \xi_m$. There is $j$ such that $\Delta$ and $\eta$ lie in different half-spaces generated by the hyperplane $L_j$; in addition, $\eta \notin L_j$. We set $\zeta=L_j\cap [\xi_j, \, \eta]$, $\Delta_j={\rm conv}\, \{\xi_1, \, \dots, \, \xi_{j-1}, \, \xi_{j+1}, \, \dots, \, \xi_m\}$. The point $b\in {\rm int}\, \Delta_j$ is defined by the condition $a\in [b, \, \xi_j]$. For $i\in \{1, \, \dots, \, m\} \backslash \{j\}$ we denote $\Delta_{ij}= {\rm conv}\, (\{\xi_1, \, \dots, \, \xi_m\}\backslash \{\xi_i, \, \xi_j\})$, and by $L_{ij}$, the affine hull of $\{\xi_1, \, \dots, \, \xi_m\}\backslash \{\xi_i, \, \xi_j\}$. Notice that for each $i\in \{1, \, \dots, \, m\}\backslash \{j\}$ the points $\{\xi_l\}_{l\ne i, \, j}$, $\zeta$ are affinely independent; otherwise, $\zeta\in L_{ij}$ and $\eta \in L_i$, which contradicts the conditions of Lemma. By the induction hypothesis, there is $i\in \{1, \, \dots, \, m\} \backslash \{j\}$ such that $b\in {\rm conv}\, (\{\zeta\}\cup \Delta_{ij})$. Hence there is $c\in \Delta_{ij}$ such that $b\in [c, \, \zeta]$. This implies that $a \in {\rm conv}\, \{\zeta, \, c, \, \xi_j\}\subset {\rm conv}\, \{\eta, \, c, \, \xi_j\}\subset {\rm conv}\, \{\eta, \, \Delta_i\}$.
\end{proof}

By generality of position (see Definition \ref{gen_pos_def}) and Definition \ref{nm_def}, the simplex $\Sigma_I:=\{(\overline{\xi})_I:\; \overline{\xi}\in \Sigma\}$ and the point $(1/\overline{p}_\beta)_I$ satisfy the conditions of Lemma \ref{conv_hull}, and $(1/q, \, \dots, \, 1/q)$ is the interior point of $\Sigma_I$. Hence there is $i\in \{1, \, \dots, \, m\}$ such that $$(1/q, \, \dots, \, 1/q)\in {\rm conv}\, \{(1/\overline{p}_{\alpha_1})_I, \, \dots, \, (1/\overline{p}_{\alpha_{i-1}})_I, \, (1/\overline{p}_{\alpha_{i+1}})_I, \, \dots, \, (1/\overline{p}_{\alpha_m})_I, \, (1/\overline{p}_{\beta})_I\}.$$

Without loss of generality, $i=m$. We set $$\overline{\gamma} = (\gamma_1, \, \dots, \, \gamma_{m-1}, \, \gamma_m)=(\alpha_1, \, \dots, \, \alpha_{m-1}, \, \beta).$$ There are numbers $\mu_j> 0$, $j=1, \, \dots, \, m$, such that $\sum \limits _{j=1}^m \mu_j=1$ and
\begin{align}
\label{q_conv_beta} \frac{1}{\theta_l(\overline{\gamma})}:=\frac 1q = \sum \limits _{j=1}^m \frac{\mu_j}{p_{\gamma_j,l}}, \quad l\in I
\end{align}
(the strict positivity follows from the generality of position); hence $\overline{\gamma} \in {\cal N}_m$.
For $i\notin I$ we set $\frac{1}{\theta_l (\overline{\gamma})}:= \sum \limits _{j=1}^m \frac{\mu_j}{p_{\gamma_j,l}}$.

We define the numbers $t_i$, $i=1, \, \dots, \, d$, as follows. For $i\notin I$ we set 
\begin{align}
\label{t_inii} t_i = \begin{cases} k_i, & i \in T_+, \\ 1, & i\in T_-. \end{cases}
\end{align}
For $i\in I$ the numbers $t_i$ are defined by the system of equations
\begin{align}
\label{t_i_def} \left\{\begin{array}{l} \frac{\nu_{\alpha_2}}{\nu_{\alpha_1}} = \prod _{i\in T_+} k_i^{1/p_{\alpha_2,i}-1/p_{\alpha_1,i}} \prod _{i\in I} t_i^{1/p_{\alpha_2,i}-1/p_{\alpha_1,i}}, \\ \dots \\ \frac{\nu_{\alpha_{m-1}}}{\nu_{\alpha_1}} = \prod _{i\in T_+} k_i^{1/p_{\alpha_{m-1},i}-1/p_{\alpha_1,i}} \prod _{i\in I} t_i^{1/p_{\alpha_{m-1},i}-1/p_{\alpha_1,i}}, \\
\frac{\nu_\beta}{\nu_{\alpha_1}} = \prod _{i\in T_+} k_i^{1/p_{\beta,i}-1/p_{\alpha_1,i}} \prod _{i\in I} t_i^{1/p_{\beta,i}-1/p_{\alpha_1,i}}.
 \end{array}\right.
\end{align}
By generality of position, the points $(1/\overline{p}_{\alpha_1})_I, \, \dots$, $(1/\overline{p}_{\alpha_{m-1}})_I$, $(1/\overline{p}_{\beta})_I$ are affinely independent; hence the numbers $t_i$ are well-defined.

Let $\mu_{j,k}\ge 0$, $1\le j\le m$, $k=1, \, 2$, $\sum \limits _{j=1}^m \mu_{j,k}=1$. We define the vectors $\overline{\theta}_k$ and the numbers $\nu_{(k)}$ by the equations
$$
\frac{1}{\overline{\theta}_k} =\sum \limits _{j=1}^m \frac{\mu_{j,k}}{\overline{p}_{\gamma_j}}, \quad \nu_{(k)} = \prod _{j=1}^m \nu_{\gamma_j}^{\mu_{j,k}}.
$$
From \eqref{t_i_def} it follows that
\begin{align}
\label{nu1_nu21} \frac{\nu_{(1)}}{\nu_{(2)}} = \frac{\prod _{l\in T_+} k_l^{1/\theta_{1,l}} \prod _{l\in I} t_l^{1/\theta_{1,l}}}{\prod _{l\in T_+} k_l^{1/\theta_{2,l}} \prod _{l\in I} t_l^{1/\theta_{2,l}}} \stackrel{\eqref{t_inii}}{=} \frac{\prod _{l=1}^d t_l^{1/\theta_{1,l}}}{\prod _{l=1}^d t_l^{1/\theta_{2,l}}}.
\end{align}
In particular, this implies that \eqref{incl_mg1_ineq} is equivalent to
\begin{align}
\label{incl_mg1_ineq1} \begin{array}{c}\nu_{\alpha_1}^{\lambda_1 (\overline{\alpha})} \dots \nu_{\alpha_m}^{\lambda_m (\overline{\alpha})} s_1^{1/p_{\beta,1}-1/\theta_1 (\overline{\alpha})} \dots s_d^{1/p_{\beta,d}-1/\theta_d(\overline{\alpha})} \le \\ \le \nu_{\alpha_1}^{\mu_1} \dots \nu_{\alpha_{m-1}}^{\mu_{m-1}}\nu_\beta^{\mu_m} t_1^{1/p_{\beta,1}-1/\theta_1 (\overline{\gamma})} \dots t_d^{1/p_{\beta,d}-1/\theta_d(\overline{\gamma})}. \end{array}
\end{align}

We set 
$$
\omega = \sup \Bigl\{ \mu \in [0, \, 1]: \, \frac{1-\mu} {\theta_i(\overline{\alpha})} + \frac{\mu} {\theta_i(\overline{\gamma})} < \frac 1q, \; i\in T_+, \; \frac{1-\mu} {\theta_i(\overline{\alpha})} + \frac{\mu} {\theta_i(\overline{\gamma})} > \frac 1q, \; i\in T_-\Bigr\}.
$$
Then $\omega>0$.

We define the vector $\overline{\sigma} = (\sigma_1, \, \dots, \, \sigma_d)$ by the equation $\frac{1}{\overline{\sigma}} = \frac{1-\omega}{\overline{\theta}(\overline{\alpha})} + \frac{\omega}{\overline{\theta}(\overline{\gamma})}$.

If $\omega=1$, we get $\sigma_i>q$ for $i\in T_+$, $\sigma_i<q$ for $i\in T_-$ (by generality of position). If $\omega < 1$, one and only one of these inequalities becomes an equality, the remaining inequalities are strict; hence $(\alpha_1, \, \dots, \, \alpha_m, \, \beta) \in {\cal N}_{m+1}$ (see Definitions \ref{nm_def} and \ref{gen_pos_def}). In both cases, by \eqref{phi_def}, \eqref{psi_def}, \eqref{min_mg1},
$$
\nu_{\alpha_1}^{\lambda_1(\overline{\alpha})} \dots \nu_{\alpha_m}^{\lambda_m(\overline{\alpha})} \prod _{i\in T_+} k_i^{1/q-1/\theta_i(\overline{\alpha})} \le$$$$\le \nu_{\alpha_1}^{(1-\omega)\lambda_1(\overline{\alpha})+\omega\mu_1} \dots \nu_{\alpha_{m-1}}^{(1-\omega)\lambda_{m-1}(\overline{\alpha})+\omega\mu_{m-1}} \nu_{\alpha_m}^{(1-\omega)\lambda_m(\overline{\alpha})} \nu_\beta^{\omega \mu_m} \prod _{i\in T_+} k_i^{1/q-(1-\omega)/\theta_i(\overline{\alpha})-\omega/\theta_i(\overline{\gamma})};
$$
this yields
$$
\nu_{\alpha_1}^{\lambda_1(\overline{\alpha})} \dots \nu_{\alpha_m}^{\lambda_m(\overline{\alpha})} \prod _{i\in T_+} k_i^{1/q-1/\theta_i(\overline{\alpha})} \le \nu_{\alpha_1}^{\mu_1}\dots \nu_{\alpha_{m-1}} ^{\mu_{m-1}} \nu_\beta^{\mu_m} \prod _{i\in T_+} k_i^{1/q-1/\theta_i(\overline{\gamma})}.
$$
Taking into account \eqref{s_inii}, \eqref{t_inii} and the equalities $\frac{1}{\theta_l(\overline{\alpha})} = \frac{1}{\theta_l(\overline{\gamma})} = \frac 1q$, $l\in I$ (see \eqref{sum_paj_i_q}, \eqref{theta_a_def}, \eqref{q_conv_beta}), we get that the last inequality can be written as follows:
$$
\nu_{\alpha_1}^{\lambda_1(\overline{\alpha})} \dots \nu_{\alpha_m}^{\lambda_m(\overline{\alpha})} \prod _{i=1}^d s_i^{1/q-1/\theta_i(\overline{\alpha})} \le \nu_{\alpha_1}^{\mu_1}\dots \nu_{\alpha_{m-1}} ^{\mu_{m-1}} \nu_\beta^{\mu_m} \prod _{i=1}^d t_i^{1/q-1/\theta_i(\overline{\gamma})}.
$$
Dividing into $\nu_{\alpha_j}$ ($j=1, \, \dots, \, m-1$) and applying \eqref{nu1_nu2}, \eqref{nu1_nu21}, we get
$$
\prod _{i=1}^d s_i^{1/q-1/p_{\alpha_j,i}} \le \prod _{i=1}^d t_i^{1/q-1/p_{\alpha_j,i}}, \quad 1\le j\le m-1.
$$
Hence for all $\tau_1\ge 0, \, \dots, \, \tau_{m-1}\ge 0$ such that $\sum \limits _{j=1}^{m-1} \tau_j=1$ we have
$$
\prod _{i=1}^d s_i^{1/q-\sum \limits_{j=1}^{m-1}\tau_j/p_{\alpha_j,i}} \le \prod _{i=1}^d t_i^{1/q-\sum \limits_{j=1}^{m-1}\tau_j/p_{\alpha_j,i}};
$$
since $s_i=t_i$ for $i\notin I$ (see \eqref{s_inii}, \eqref{t_inii}), this is equivalent to
\begin{align}
\label{s_i_q_tau} \prod _{i\in I} s_i^{1/q-\sum \limits_{j=1}^{m-1}\tau_j/p_{\alpha_j,i}} \le \prod _{i\in I} t_i^{1/q-\sum \limits_{j=1}^{m-1}\tau_j/p_{\alpha_j,i}}.
\end{align}
Now we choose $\tau_1, \, \dots, \, \tau_{m-1}$ such that the vectors $\left(\frac{1}{\overline{p}_\beta} - \sum \limits_{j=1}^{m-1} \frac{\tau_j}{\overline{p}_{\alpha_j}}\right)_I$ and $\left(\frac{1}{\overline{q}} - \sum \limits_{j=1}^{m-1} \frac{\tau_j}{\overline{p}_{\alpha_j}}\right)_I$ are co-directional (here $\overline{q}=(q, \, \dots, \, q)$); it is possible since $$\Bigl(\frac{1}{\overline{q}}\Bigr)_I\in {\rm conv}\Bigl\{\Bigl(\frac{1}{\overline{p}_{\alpha_1}}\Bigr)_I, \, \dots, \, \Bigl(\frac{1}{\overline{p}_{\alpha_{m-1}}}\Bigr)_I, \, \Bigl(\frac{1}{\overline{p}_{\beta}}\Bigr)_I\Bigr\}.$$ Then from \eqref{s_i_q_tau} it follows that
$$
\prod _{i\in I} s_i^{1/p_{\beta,i}-\sum \limits_{j=1}^{m-1}\tau_j/p_{\alpha_j,i}} \le \prod _{i\in I} t_i^{1/p_{\beta,i}-\sum \limits_{j=1}^{m-1}\tau_j/p_{\alpha_j,i}};
$$
again, by \eqref{s_inii}, \eqref{t_inii}, it is equivalent to
$$
\prod _{i=1}^d s_i^{1/p_{\beta,i}-\sum \limits_{j=1}^{m-1}\tau_j/p_{\alpha_j,i}} \le \prod _{i=1}^d t_i^{1/p_{\beta,i}-\sum \limits_{j=1}^{m-1}\tau_j/p_{\alpha_j,i}}.
$$
This can be written as
$$
\prod_{i=1}^d s_i^{1/p_{\beta,i}-1/\theta_i(\overline{\alpha})}\prod_{j=1}^{m-1}\left(\prod_{i=1}^d s_i^{1/\theta_i(\overline{\alpha})-1/p_{\alpha_{j},i}}\right)^{\tau_j} \le \prod_{i=1}^d t_i^{1/p_{\beta,i}-1/\theta_i(\overline{\gamma})}\prod_{j=1}^{m-1}\left(\prod_{i=1}^d t_i^{1/\theta_i(\overline{\gamma})-1/p_{\alpha_{j},i}}\right)^{\tau_j}.
$$
Taking into account that, by \eqref{nu1_nu2}, \eqref{nu1_nu21}, the equalities $\prod_{i=1}^d s_i^{1/\theta_i(\overline{\alpha})-1/p_{\alpha_{j},i}} = \frac{\nu_{\alpha_1}^{\lambda_1(\overline{\alpha})}\dots \nu_{\alpha_m}^{\lambda_m(\overline{\alpha})}}{\nu_{\alpha_j}}$, $\prod_{i=1}^d t_i^{1/\theta_i(\overline{\gamma})-1/p_{\alpha_{j},i}} = \frac{\nu_{\alpha_1}^{\mu_1}\dots \nu_{\alpha_{m-1}}^{\mu_{m-1}}\mu_\beta^{\mu_m}}{\nu_{\alpha_j}}$ hold, we get
$$\nu_{\alpha_1}^{\lambda_1(\overline{\alpha})}\dots \nu_{\alpha_m}^{\lambda_m(\overline{\alpha})}
\prod_{i=1}^d s_i^{1/p_{\beta,i}-1/\theta_i(\overline{\alpha})} \le \nu_{\alpha_1}^{\mu_1}\dots \nu_{\alpha_{m-1}}^{\mu_{m-1}}\nu_\beta^{\mu_m}\prod_{i=1}^d t_i^{1/p_{\beta,i}-1/\theta_i(\overline{\gamma})};
$$
this completes the proof of \eqref{incl_mg1_ineq1}.

\section{The lower estimate in the general case}

Let us first prove a lower estimate for finite $A$, but without the condition of general position. We argue similarly as in \cite[\S 4]{vas_mix_sev}. First we notice that if $\overline{p}=(p_1, \, \dots,\,  p_d)$, $\overline{\theta}=(\theta_1, \, \dots,\,  \theta_d)\in [1, \, \infty]^d$, $|1/p_i-1/\theta_i|\le \frac{\log 2}{\log k}$, $i=1, \, \dots, \, d$, then $\frac 12 B^{\overline{k}}_{\overline{p}} \subset B^{\overline{k}}_{\overline{\theta}} \subset 2B^{\overline{k}}_{\overline{p}}$.

Let $\alpha \in A$, $\overline{p}_{\alpha}^N \in [1, \, \infty]^d$, $1/\overline{p}_{\alpha}^N \underset{N\to \infty}{\to} \overline{p}_{\alpha}$. For $\beta \ne \alpha$ we write $\overline{p}^N_{\beta} = \overline{p}_{\beta}$.
The sets ${\cal N}_m^N$, the numbers $\lambda_j^N(\overline{\alpha}, \, I)$ and the vectors $\overline{\theta}^N(\overline{\alpha}, \, I)$ are defined according to Definition \ref{nm_def} for $\{\overline{p}^N_{\beta}\}_{\beta \in A}$. We show that
$$
\Psi(\{\overline{p}^N_{\beta}\}_{\beta\in A}, \, \overline{k}, \, q) \underset{N\to \infty}{\to} \Psi(\{\overline{p}_{\beta}\}_{\beta\in A}, \, \overline{k}, \, q)
$$
(see \eqref{psi_def}). It follows from the assertions below:
\begin{enumerate}
\item The function $\Phi$ from \eqref{phi_def} is continuous in $1/\overline{p}$.

\item Let $\overline{\alpha} = (\alpha_1, \, \dots, \, \alpha_m)\in {\cal N}_m$, $\alpha_1=\alpha$. Then for sufficiently large $N$ we have $\overline{\alpha}\in {\cal N}_m^N$, $\lambda_j^N(\overline{\alpha}, \, I) \underset{N \to \infty}{\to} \lambda_j(\overline{\alpha}, \, I)$, $j=1, \, \dots, \, m$, $1/\overline{\theta}^N(\alpha, \, I) \underset{N \to \infty}{\to} 1/\overline{\theta}(\alpha, \, I)$.

\item Let $\overline{\alpha} = (\alpha_1, \, \dots, \, \alpha_m)\notin {\cal N}_m$, $\alpha_1=\alpha$. Suppose that there is a subsequence $\{N_j\}_{j\in \N}$ such that $\overline{\alpha} \in {\cal N}_m^{N_j}$, $j\in \N$, and the set $I$ from Definition \ref{nm_def} is the same for all $j$. Then there is a nonempty subset $J=\{i_1, \, \dots, \, i_s\} \subset \{1, \, \dots, \, m\}$ such that $\overline{\alpha}_J := (\alpha_{i_1}, \, \dots, \, \alpha_{i_s})\in {\cal N}_s$, $i_1=1$, and
$$
\nu_{\alpha_1}^{\lambda_1^{N_j}(\overline{\alpha}, \, I)} \dots \nu_{\alpha_m}^{\lambda_m^{N_j}(\overline{\alpha}, \, I)} \Phi (\overline{\theta}^{N_j}(\overline{\alpha}, \, I), \, \overline{k}, \, q) \underset{j\to \infty}{\to} \nu_{\alpha_{i_1}}^{\lambda _{i_1}(\overline{\alpha}_J, \, \tilde I)} \dots \nu_{\alpha_{i_s}}^{\lambda _{i_s}(\overline{\alpha}_J, \, \tilde I)} \Phi (\overline{\theta}(\overline{\alpha}_J, \, \tilde I), \, \overline{k}, \, q),
$$
where $\tilde I\subset I$ is from Definition \ref{nm_def} for $\overline{\alpha}_J \in {\cal N}_s$.
\end{enumerate}

It remains, by translating in succession the vectors $1/\overline{p}_{\alpha}$, $\alpha\in A$, to obtain a set of vectors in general position lying sufficiently close to $\{(1/\overline{p}_{\alpha})\}_{\alpha\in A}$.

A transition from a finite to an arbitrary set $A$ proceeds as in \cite[\S 5]{vas_mix_sev}.

\end{document}